\documentclass[journal]{ieeetran}
\usepackage{cite}
\usepackage{graphicx}
\usepackage{epstopdf}
\usepackage{amsmath}
\usepackage{cases}
\usepackage[caption=false,font=footnotesize]{subfig}
\usepackage{fixltx2e}
\usepackage{amssymb}
\usepackage{mathtools}
\usepackage{bm}
\usepackage{multirow}
\usepackage{color}
\usepackage[dvipsnames]{xcolor}
\usepackage{algorithm}
\usepackage{algorithmic}
\usepackage{amsthm}
\usepackage{mathrsfs}
\usepackage{textcomp,mathcomp}
\usepackage{booktabs}
\usepackage{threeparttable}
\newtheorem{theorem}{Theorem}

\usepackage{pifont}
\usepackage{threeparttable}
\usepackage{float}

\newtheorem{definition}{Definition}
\newcommand{\tabincell}[2]{\renewcommand\arraystretch{0.9}\begin{tabular}{@{}#1@{}}#2\end{tabular}}
\begin{document}


\title{\LARGE{Hot Standby in Ammonia Synthesis Reshapes Market Equilibrium in Renewable P2A Systems: A Potential Game Approach}}

\author{
Yangjun~Zeng,~\IEEEmembership{Student Member,~IEEE},
Yiwei~Qiu,~\IEEEmembership{Member,~IEEE},
Xiaocong~Sun,
Jie~Zhu,~\IEEEmembership{Student Member,~IEEE},
Jiarong~Li,~\IEEEmembership{Member,~IEEE},
Shi~Chen,~\IEEEmembership{Member,~IEEE},
Buxiang~Zhou,~\IEEEmembership{Member,~IEEE},
and
Kaigui~Xie,~\IEEEmembership{Fellow,~IEEE}
%

\thanks{Y. Zeng, Y. Qiu, J. Zhu, S. Chen, B. Zhou, and K. Xie are with the College of Electrical Engineering, Sichuan University, Chengdu 610065, China. X. Sun is with the National Power Dispatching and Control Center, State Grid Corporation of China, Beijing 100031, China. J. Li is with the Harvard John A. Paulson School of Engineering and Applied Sciences, Harvard University, Cambridge 02138, USA. \emph{(Corresponding author: Yiwei Qiu)}
}%
}
\maketitle

\begin{abstract}
Integrating renewable generation, hydrogen production, and renewable ammonia (RA) synthesis into power-to-ammonia (P2A) systems creates interactions across electricity and hydrogen markets. Limited operational flexibility, however, places RA at a disadvantage at the Nash equilibrium (NE). Recent advances in ammonia synthesis reactor design enable hot standby (HSB) operation, improving flexibility but introducing integer decision variables that complicate market equilibrium analysis. To address this challenge, we develop a potential game model and derive a convergent $\epsilon$-approximate equilibrium via an iterative best-response approach. Case studies show that HSB reduces RA's reliance on hydrogen purchases and increases its profit by 20.14\%. More importantly, HSB shifts the market equilibrium toward a more mutually beneficial outcome.
\end{abstract}


\section{Introduction}
\label{sec:intro}

\IEEEPARstart{I}{ntegrating} renewable generation (RG), hydrogen production (HP), and downstream renewable ammonia (RA) synthesis into renewable power-to-ammonia (ReP2A) systems offers a promising pathway for decarbonizing the power, transport, and chemical sectors \cite{zeng2025planning, yu2023optimal,chyong2025energy}. Because these stakeholders exchange electricity and hydrogen while pursuing competing interests, their interaction forms a noncooperative game \cite{zeng2025planning}.

This conflict is further amplified by a supply-demand mismatch. Renewable hydrogen supply is inherently variable, whereas conventional ammonia synthesis (ASY) requires continuous operation with limited load flexibility \cite{yu2023optimal}. 
Consequently, RA is disadvantaged in market interactions and attains low payoffs at the Nash equilibrium (NE), which weakens participation incentives \cite{zeng2025planning} and reduces renewable energy utilization \cite{chyong2025energy}.

Recent advances in ASY reactor design enable hot standby (HSB) operation, which improves 
flexibility and may alter the market equilibrium \cite{chengda2024ammonia, wu2025dispatchable}. However, operation-standby switching introduces binary variables and makes the game nonconvex. Conventional methods based on Karush-Kuhn-Tucker (KKT) conditions and iterative diagonalization may fail or lack convergence guarantees \cite{chen2024approaching}. Although Tohidi \textit{et al.} \cite{tohidi2016sequential} proposed an enumeration-based Moore-Bard algorithm that can identify the NE, its computational cost is prohibitive.

To address these challenges, we reformulate the problem as a potential game (PG) that captures ASY multi-state switching. An iterative best-response (BR) algorithm is developed to solve the NE with guaranteed convergence. Convex relaxation further enables an $\epsilon$-approximate NE ($\epsilon$-NE) without enumerating binary variables. Results show that HSB reshapes the market equilibrium, reduces RA’s reliance on hydrogen purchases, and improves its profitability.

\begin{figure}[t]
  \centering
  \includegraphics[width=3.52in]{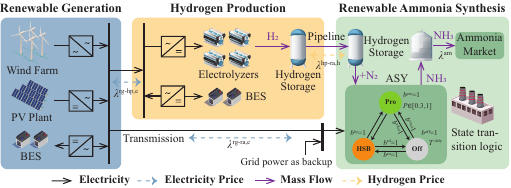}\vspace{-6pt}
  \caption{Schematic a typical multistakeholder ReP2A system.}
  \label{fig:system}
\end{figure}

\section{Game Formulation of the ReP2A System}

\subsection{Multistakeholder Structure and Transactions}
\label{sec:structure}

A typical ReP2A system is shown in Fig. 1. RG supplies electricity to HP and RA at  $P_{t}^{\text{rg,sell,hp}}$ and $P_{t}^{\text{rg,sell,ra}}$. HP produces electrolytic hydrogen and sells $f_t^{\text{hp,sell,ra}}$ to RA. RA consumes electricity and hydrogen to produce ammonia. 
Electricity and hydrogen transactions satisfy the market-clearing conditions
\begin{gather}
	P_{t}^{\text{rg,sell,hp/ra}}-P_t^{\text{hp/ra,buy,rg}}=0:\lambda_t^{\text{rg-hp/ra,e}}, \label{eq:market1}\\
	f_t^{\text{hp,sell,ra}}-f_t^{\text{ra,buy,hp}}=0:\lambda_t^{\text{hp-ra,h}}, \label{eq:market3}
\end{gather}
\noindent
where $P_t^{\text{hp/ra,buy,rg}}$ and $\lambda_t^{\text{rg-hp/ra,e}}$ denote renewable electricity purchases by HP/RA and prices; and $f_t^{\text{ra,buy,hp}}$ and $\lambda_t^{\text{hp-ra,h}}$ denote hydrogen purchase by RA and the price.

\subsection{Multi-State Switching of the ASY}
\label{sec:multistate}

Unlike conventional ASY designed for continuous operation, advances in reactor structure, insulation, and process flowsheets enable HSB to handle intermittent hydrogen supply \cite{chengda2024ammonia}.

With the HSB ability, ASY operates in three states as shown in Fig. \ref{fig:system}: \emph{Production} $b_{t}^{\text{pro}}$, \emph{HSB} $b_{t}^{\text{by}}$, and \emph{Idle} $b_{t}^{\text{off}}$. In \emph{Production}, the load ranges from $\underline{\eta}^{\text{asy}}=30\%$ to $\overline{\eta}^{\text{asy}}=100\%$ \cite{yu2023optimal,zeng2025planning,wu2025dispatchable}. In \emph{HSB}, 
a small amount of electricity maintains reactor temperature and pressure for fast restart. Switching from \emph{Idle} to \emph{HSB} or \emph{Production} requires several hours. The three states satisfies (\ref{eq:logic}) with \emph{Startup} $b_{t}^{\text{su}}$ and \emph{Shutdown} $b_{t}^{\text{sd}}$ actions defined in (\ref{eq:start})--(\ref{eq:shut}), and a minimal downtime $T^{\text{dely}}$ enforced by (\ref{eq:dely}).
\begin{gather}
	b_{t}^{\text{pro}}+b_{t}^{\text{by}}+b_{t}^{\text{off}}=1, \label{eq:logic}\\
	b_{t}^{\text{pro}}+b_{t}^{\text{by}}+b_{t-1}^{\text{off}}-1\leq b_{t}^{\text{su}}, \label{eq:start}\\
	b_{t-1}^{\text{pro}}+b_{t-1}^{\text{by}}+b_{t}^{\text{off}}-1\le b_{t}^{\text{sd}}, \label{eq:shut}\\
	 \sum\nolimits_{\tau=t}^{t+T^{\text{dely}}-1} b_{\tau}^{\text{off}}\geq T^{\text{dely}}(b_{t}^{\text{off}}-b_{t-1}^{\text{off}}). \label{eq:dely}
\end{gather}

\subsection{Mathematical Modeling of Operational Decisions}

Each stakeholder $k \in \{\text{rg,hp,ra}\}$ determines its operation and transactions to minimize cost $C_k$. Their operational models combine our prior work \cite{zeng2025planning} and (\ref{eq:logic})--(\ref{eq:dely}), summarized as
\begin{align}
	& (\text{RG})~~\min~\sum\nolimits_{t=1}^T\big(-P_{t}^{\text{rg,sell,hp}}\lambda_t^{\text{rg-hp,e}}-P_{t}^{\text{rg,sell,ra}}\lambda_t^{\text{rg-ra,e}} \hspace{12pt} \nonumber 
\end{align}

~
\vspace{-17.5pt}
\begin{align}	
	&~~~~~~~~~~~~~~~+\sigma^{\text{deg}}P_{t}^{\text{rg,bes,d}}\big)\Delta t, \label{eq:objrg} \\
	&\text{s.t.~~battery energy storage operation: (2)--(6) in \cite{zeng2025planning}}, \\
	& \text{~~~~~electrical network power flow: (7)--(16) in \cite{zeng2025planning}}. \\
	& (\text{HP})~~\min~\nonumber  \sum\nolimits_{t=1}^T \big(P_t^{\text{hp,buy,rg}}\lambda_t^{\text{rg-hp,e}}+\sigma^{\text{deg}}P_{t}^{\text{hp,bes,d}}  \\
	&~~~~~~~~~~~~~~~-f_t^{\text{hp,sell,ra}}\lambda_t^{\text{hp-ra,h}}\big)\Delta t \label{eq:objhp}\\
	&\text{s.t.~~battery energy storage operation}, \\
	&\text{~~~~~hydrogen production (19)--(22),  storage (23)--(26), } \nonumber\\
	&\text{~~~~ and delivery (27)--(33) in \cite{zeng2025planning}}, \\
	& (\text{RA})~~ \min~\nonumber \sum_{t=1}\nolimits^T \big[(f_t^{\text{ra,buy,hp}}\lambda_t^{\text{hp-ra,h}}+P_t^{\text{ra,buy,rg}}\lambda_t^{\text{rg-ra,e}}\\
	&~~~~~~~~~~~~~~~+P_t^{\text{ra,back}}\lambda^{\text{ra,back}}-M_t^{\text{ra,sell}} \lambda_t^{\text{am}})\Delta t+ c^{\text{su}} b_{t}^{\text{su}} \big] \label{eq:objra}\\
	&\text{s.t.~~hydrogen storage operation}, \\
	&\text{~~~~~ASY operation: (36)--(37), (41) in \cite{zeng2025planning} and (\ref{eq:logic})--(\ref{eq:dely})}, \label{eq:asyop}\\
	&~~~~~P_{t}^{\text{ra,back}}+P_{t}^{\text{ra,buy,rg}}=P_{t}^{\text{ra,asy}}+b_{t}^{\text{by}} P^{\text{ra,by}},  \label{eq:pas} \\
	&~~~~~b_{t}^{\text{pro}} \underline{\eta}^{\text{asy}}W^{\text{ra,asy}}\leq M_{t}^{\text{ra,pro}}\leq b_{t}^{\text{pro}} \overline{\eta}^{\text{asy}}W^{\text{ra,asy}}, \label{eq:aslimit}\\
	&~~~~~|M_{t}^{\text{ra,pro}}-M_{t-1}^{\text{ra,pro}}|\leq\overline{r}^{\text{asy}}W^{\text{ra,asy}}, \label{eq:asramp}    \\
	&\text{~~~~~ammonia storage operation: (42)--(45) in \cite{zeng2025planning} },
\end{align}
\noindent
where $T$ and $\Delta t$ are the operational horizon and step length;
$P_{t}^{\text{rg/hp,bes,d}}$ and $\sigma^{\text{deg}}$ are battery charging/discharging power and degradation cost;
$c^{\text{su}}$ is the ASY startup cost;
$P_{t}^{\text{ra,back}}$ and $\lambda^{\text{ra,back}}$ are backup power and cost;
$M_{t}^{\text{ra,pro}}$, $M_t^{\text{ra,sell}}$, and  $\lambda_t^{\text{am}}$ are ammonia production, sales, and price;
$P^{\text{ra,by}}$ is the HSB power; $W^{\text{ra,asy}}$ and $\overline{r}^{\text{asy}}$ are the ASY capacity and ramping limit.

Let ${x}_k$ denote the decision variables of stakeholder $k$, and ${x}_{-k}$ those of others. The feasible strategy set of stakeholder $k$ is denoted by $X_k({x}_{-k})$. The market-clearing conditions are denoted by $\varphi({x}_k,{x}_{-k})$. The multistakeholder game of the ReP2A system is thus formulated as:
\begin{align}
	(G) \ \min_{{x}_k}~C_k,~~\text{s.t. } {x}_k \in X_k({x}_{-k}),~ \forall k. \label{eq:game}
\end{align}

Due to binary variables, $G$ is nonconvex, and its NE is difficult to compute \cite{chen2024approaching}.

\section{Solution Method for the Equilibrium}

\subsection{Construction of Potential Game}

To compute the NE, we reformulate the game $G$ as a PG and solve it via iterative BR, which ensures convergence \cite{cenedese2019charging}.

\begin{definition}
	A finite game is a PG if there exists a potential function $\Phi(x)$ such that, for any stakeholder $k$, any strategies $x_k, x_k' \in X_k$, and fixed $x_{-k}$, we have
	\begin{align}
		\hspace{-3pt}\Phi(x_k, x_{-k})-\Phi(x_k', x_{-k})=J_k(x_k, x_{-k})-J_k(x_k', x_{-k}). \hspace{-5pt}
	\end{align}
\end{definition}

Here, we construct $J_k$ and $\Phi$ as 
\begin{gather}
	J_k= C_k+\frac{\rho}{2} \Vert \varphi({x}_k,{x}_{-k})\Vert_2^2, \\
	\Phi =\sum \nolimits_k C_k+\frac{\rho}{2} \Vert \varphi({x}_k,{x}_{-k})\Vert_2^2,
\end{gather}
\noindent where $\rho$ is a penalty factor. The quadratic penalty on $\varphi$ (with its linear term included in $C_k$) enforces the market-clearing conditions (\ref{eq:market1})--(\ref{eq:market3}) within an ADMM framework and stabilizes price updates \cite{he2026fully}. At equilibrium, $J_k$ equals $C_k$.

In a PG, the potential function is nonincreasing under iterative BR, with each stakeholder $k$ acting following $BR_k({x}_{-k}):={\arg\min}_{{x}_k \in X_k({x}_{-k})} \ J_k$, which ensures convergence \cite{cenedese2019charging}.



\subsection{$\epsilon$-Approximate NE and Iterative BR Solution Algorithm}

Although convergence is ensured in PG, binary variables in the ASY operation (\ref{eq:asyop})--(\ref{eq:asramp}) increase BR computation cost.
To improve tractability, we relax the binary variables to $[0,1]$ at each iteration, yielding a relaxed strategy set $X_{r,k}$. The relaxed solution is then rounded to recover a feasible strategy, serving as an approximate BR with error 
\begin{align}
	\epsilon^i_k=J_{k}(x^{i+1}_k, x^{i}_{-k})-\inf_{x_k \in X_k(x^i_{-k})}J_k(x_k, x^{i}_{-k})\leq \overline{\epsilon}.
\end{align}
\begin{definition}
	A strategy profile $\{\hat{x}_k, \hat{x}_{-k}\}$ is an $\epsilon$-NE if
	\begin{align}
		J_k(\hat{x}_k, \hat{x}_{-k})\leq \inf_{{x}_k \in X_k(x_{-k})}J_k({x}_k, \hat{x}_{-k})+\epsilon,~\forall k. 
	\end{align}
\end{definition}
The solution quality is measured by
\begin{align}
	d^{i}_k:=J_{k}(x^{i}_k, x^{i}_{-k})-\inf_{x_k \in X_k(x^i_{-k})}J_k(x_k, x^{i}_{-k}) \geq 0.
\end{align}
\noindent 
where smaller $d^{i}_k$ indicates a solution closer to the NE. Theorem \ref{theorem:1} ensures convergence of the relaxed BR process.

\begin{algorithm}[t]
	\small	
	\caption{ \small Iterative BR Algorithm for solving the NE}
	\label{alg:BR}
		\begin{algorithmic}[1]
		\STATE \textbf{precondition:} Initialize maximum iteration $i^m$, price $\lambda^0$, and penalty $\rho$, and tolerance $g_1$ and $g_2$ \label{line:1}
		\STATE  Solve (\ref{eq:game}) for all $k$ to obtain initial strategies $\{x_k^0,x_{-k}^0\}$ \label{line:3}
		\STATE Set iteration index $i\gets0$, and compute $\Phi^0$ and $\varphi^0$
		\FOR {$i=0$ \TO $i^m-1$}
		\STATE Update RG's BR ${x}^{i+1}_\text{rg} \gets {\arg\min}_{{x}_\text{rg} \in X_{\text{rg}}({x}^{i}_\text{hp},{x}^{i}_\text{ra})} \ J_\text{rg}$
		\STATE  Update HP's BR ${x}^{i+1}_\text{hp} \gets {\arg\min}_{{x}_\text{hp} \in X_{\text{hp}}({x}^{i+1}_\text{rg},{x}^{i}_\text{ra})} \ J_\text{hp}$
		\STATE Solve relaxed $\tilde{x}^{i+1}_\text{ra} \gets {\arg\min}_{{x}_\text{ra} \in X_{r,\text{ra}}({x}^{i+1}_\text{rg},{x}^{i+1}_\text{hp})} \ J_\text{ra}$  \label{line:7}
		\STATE Get RA's approximate BR ${x}^{i+1}_\text{ra} \gets \tilde{x}^{i+1}_\text{ra}|_{b=\lfloor \tilde{b}^{i+1}\rceil}$ \label{line:8}
		\STATE Compute $\Phi^{i+1}$ and $\varphi^{i+1}$
		\STATE Update $\lambda^{i+1} \gets \lambda^{i}+\rho \varphi^{i+1}$
		\IF{$|\Phi^{i+1}-\Phi^{i}|<g_1$ and $\Vert\varphi^{i+1}\Vert_2<g_2$}
		\STATE \textbf{break}
		\ENDIF
		\ENDFOR
		\RETURN $\{\lambda^{i+1},{x}^{i+1}_{k,\forall k}\}$
	\end{algorithmic}
\end{algorithm}

\begin{theorem}
	\label{theorem:1}
In the PG, $\limsup_{i\to \infty} \max_{k\in \text{\{rg,hp,ra\}}}d^{i}_k \leq \overline{\epsilon}$, and iterates asymptotically lie in the $\overline{\epsilon}$-NE set.
\end{theorem}

\begin{proof}
	From the properties of PGs, we have
	\begin{align}
		&\nonumber \Phi(x^{i}_k, x^{i}_{-k})-\Phi(x^{i+1}_k, x^{i}_{-k})=J_{k}(x^{i}_k, x^{i}_{-k})-J_{k}(x^{i+1}_k, x^{i}_{-k})\\
		&\nonumber =J_{k}(x^{i}_k, x_{-k})-\inf_{x_k}J_k(x_k, x^{i}_{-k}) +\inf_{x_k}J_k(x_k, x^{i}_{-k}) \\
		&~~~~- J_{k}(x^{i+1}_k, x_{-k}) = d^{i}_k-\epsilon^i_k \geq  d^{i}_k-\overline{\epsilon}. \label{eq:proof}
	\end{align}
	Since $\Phi$ is monotonic, $\limsup_{i\to \infty}\ (\ref{eq:proof}) = 0$.
\end{proof}

By controlling $\epsilon^i_k$, the algorithm converges to a high-quality $\epsilon$-NE. 
The procedure is summarized in Algorithm \ref{alg:BR}. 
The approximate BR is accepted if $J_\text{ra}({x}^{i+1}_\text{ra})-J_\text{ra}(\tilde{x}^{i+1}_\text{ra})\leq \overline{\epsilon}$; otherwise, the exact BR is solved to ensure convergence.

\section{Case Study}
\label{sec:cases}

Based on a real-life  ReP2A project in Inner Mongolia, China \cite{yu2023optimal, zeng2025planning}, we compare three settings: \textbf{M1}, where ASY has no HSB capability \cite{zeng2025planning}; \textbf{M2}, where ASY has HSB capability; and \textbf{M3}, a cooperative benchmark in which ASY has HSB capability and inter-stakeholder decisions are coordinated to represent the global optimum. System parameters are taken from \cite{zeng2025planning} and \cite{zeng2026carbon}. The 7-day wind and solar profiles for two scenarios, together with the ammonia price, are shown in Fig. \ref{fig:WT}.

\begin{figure}[t]
	\centering
	\includegraphics[width=3.35in]{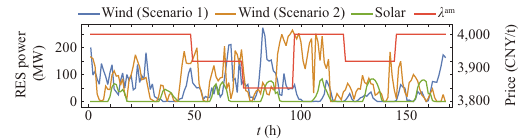}\vspace{-8pt}
	\caption{Scenarios of wind and solar power generation, and ammonia price.}
	\label{fig:WT}\vspace{5pt}
	\includegraphics[width=3.35in]{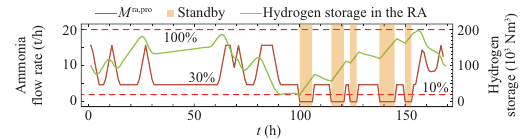}\vspace{-8pt}
	\caption{Ammonia yield and hydrogen storage in the RA (Scenario 1, M2).}
	\label{fig:ASY}\vspace{5pt}
	\includegraphics[width=3.35in]{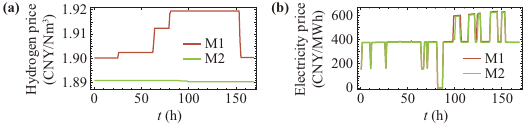}\vspace{-8pt}
	\caption{(a) Hydrogen, and (b) electricity prices in Scenario 1.}
	\label{fig:price}\vspace{5pt}
	\includegraphics[width=3.35in]{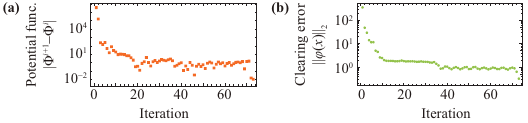}\vspace{-8pt}
	\caption{Convergence of the iterative BR solution algorithm. (a) Incremental potential function. (b) Norm of market clearing residual.}
	\label{fig:convergence}
\end{figure}

In Scenario 1, where renewable energy sources (RESs) are scarce, Fig. \ref{fig:ASY} shows RA operation under M2, and Fig. \ref{fig:price} compares electricity and hydrogen prices under M1 and M2. Results are summarized in Table \ref{tab:comparison}. Under M1, RA incurs a loss, confirming its weak position and lack of participation incentives in the noncooperative game \cite{zeng2025planning}. With HSB enabled in M2, RA  switches to standby during low supply, such as $t\in[100,106]$ h, reducing hydrogen purchases and lowering hydrogen prices. RA therefore becomes profitable. This demonstrates that HSB improves RA's resilience under weak RES supply and restores participation incentives.

In Scenario 2, with a stronger renewable supply, HSB is used less frequently, and the profit gain is smaller,
indicating that HSB's value is greater under limited or volatile conditions. Further comparison between M2 and the cooperative benchmark M3 shows that the $\overline{\epsilon}$-NE closely approximates the global optimum. This not only validates the solution method but also shows that HSB moves the noncooperative game among RG, HP, and RA toward a more mutually beneficial outcome.

Table \ref{tab:comparison2} further evaluates the effect of HSB over 12 typical weeks. The average RA profit increases from 1.44 $\times 10^6$ CNY under M1 to 1.73 $\times 10^6$ CNY under M2, a gain of 20.14\%. The profits of RG and HP change only slightly, indicating that the main benefit of HSB is to strengthen RA's economic position rather than to change the system symmetrically.

In terms of computation, $\overline{\epsilon}$ is set to 0.01\% of the potential function $\Phi$. Fig. \ref{fig:convergence} shows convergence of the iterative BR algorithm. All computations are performed on a laptop with an \emph{Intel Core Ultra 7 165H@3.80 GHz} CPU and 32 GB of RAM, and convergence is reached within 9 min.

\begin{table}[t]\scriptsize
	\renewcommand{\arraystretch}{1.2}
	\caption{Operational Comparison under Different Settings}\vspace{-7pt}
	\label{tab:comparison}
	\centering
	\begin{tabular}{c@{\hspace{4pt}}c@{\hspace{4pt}}c@{\hspace{4pt}}c@{\hspace{4pt}}c}
		\toprule
		\tabincell{c}{Settings}    & \tabincell{c}{Revenue \{RG, HP, RA\}\\(10$^4$ CNY)}    & \tabincell{c}{Total revenue\\(10$^4$ CNY)}  & \tabincell{c}{ Average price\\ \{$\lambda^{\text{rg-hp/ra,e}}$, $\lambda^{\text{hp-ra,h}}$\}\\(CNY/MWh, CNY/Nm$^3$)}  \\
		\midrule
		M1 (Scenario 1)       &  \{345.12, 55.97, -2.41\}            & 398.68  & \{3.244, 1.909\}        \\
		M2 (Scenario 1)       &  \{341.26, 54.44, \textbf{3.33}\}    & 399.01  &\{3.199, 1.891\}        \\
		M3 (Scenario 1)       &        /                             & 400.30  &/        \\
		\hline
		M1 (Scenario 2)       &  \{370.33, 180.94, 7.67\}     		 &  558.94 &  \{2.502, 1.901\}           \\
		M2 (Scenario 2)       &  \{369.12, 181.35, \textbf{8.69}\}   & 559.16  &  \{2.497, 1.899\}  \\
		M3 (Scenario 2)       &        /                             & 559.65  &/        \\
		\bottomrule
	\end{tabular}
\end{table}

\begin{table}[t]\scriptsize
	\renewcommand{\arraystretch}{1.2}
	\caption{Revenue Comparison Between M1 and M2 over 12 Typical Weeks}\vspace{-7pt}
	\label{tab:comparison2}
	\centering
	\begin{tabular}{c@{\hspace{5pt}}c@{\hspace{5pt}}c@{\hspace{5pt}}c@{\hspace{5pt}}c}
		\toprule
		Settings    & RG (10$^6$ CNY)    &HP (10$^6$ CNY)   & RA (10$^6$ CNY)    \\
		\midrule
		M1 (w/o HSB)    	&  39.43  	&  31.51    &   1.44      \\
		M2 (w/ HSB)         &  39.34  	& 31.35  & \textbf{1.73 (+20.14\%)}     \\
		\bottomrule
	\end{tabular}
\end{table}



\section{Conclusion}
\label{sec:conclusion}

This letter examines how  ASY's HSB capability reshapes the equilibrium of the noncooperative ReP2A system. To address the nonconvexity introduced by production-standby switching, a potential-game approach is developed, and an $\epsilon$-approximate equilibrium is obtained via an iterative BR scheme. Main findings include:

1) HSB improves the flexibility of RA, reduces its dependence on hydrogen purchase, and increases its profitability by 20.14\%;

2) The benefit is most pronounced when the renewable supply is limited or volatile;

3) More interestingly, HSB helps the interaction among RG, HP, and RA move closer to a mutually beneficial outcome, strengthening incentives for ReP2A deployment.


Future work may focus on market and incentive design. Extension to dynamic and multi-timescale games that include RG, HP, RA, and ammonia consumers also merits study.

\end{document}